\documentclass[11pt,a4paper]{amsart}
\usepackage{fullpage}
\usepackage{microtype}
\usepackage[OT1]{fontenc}
\usepackage{eulervm}
\usepackage[tt=false]{libertine} 
\usepackage{amsmath}
\usepackage{amssymb}
\usepackage{amsthm}
\usepackage{mathtools} 
\usepackage{algorithm}
\usepackage{algpseudocode}
\usepackage{enumitem}
\usepackage[round]{natbib}

\usepackage[margin=1cm]{caption} 
\usepackage{subfig}

\usepackage[thinlines]{easytable}

\usepackage[bookmarks=true,hypertexnames=false,pagebackref]{hyperref}
\hypersetup{colorlinks=true, citecolor=blue, linkcolor=red, urlcolor=blue}

\usepackage{tikz}
\usetikzlibrary{arrows,arrows.meta,backgrounds,calc,fit,decorations.pathreplacing,decorations.markings,shapes.geometric}

\tikzstyle{internal} = [draw, fill, shape=circle]
\tikzstyle{external} = [shape=circle]
\tikzstyle{square}   = [draw, fill, rectangle]
\tikzstyle{triangle} = [draw, fill, regular polygon, regular polygon sides=3, inner sep=3pt]
\tikzstyle{pentagon} = [draw, fill, regular polygon, regular polygon sides=5, inner sep=2pt, minimum size=14pt]
\tikzset{every fit/.append style=text badly centered}

\usetikzlibrary{positioning,chains,fit,shapes,calc}
\usetikzlibrary{trees}
\usetikzlibrary{decorations.pathreplacing}
\usetikzlibrary{decorations.pathmorphing}
\usetikzlibrary{decorations.markings}
\tikzset{>=latex} 

\usepackage{ifthen}

\usepackage{cleveref}

\usepackage[textsize=tiny]{todonotes}
\usepackage{mleftright}

\usepackage{cool}
\Style{DSymb={\mathrm d},DShorten=true,IntegrateDifferentialDSymb=\mathrm{d}}

\newcommand{\Ex}{\mathop{\mathbb{{}E}}\nolimits}
\renewcommand{\Pr}{\mathop{\mathrm{Pr}}\nolimits}

\def\*#1{\mathbf{#1}}
\def\+#1{\mathcal{#1}}
\def\-#1{\mathrm{#1}}
\def\=#1{\mathbb{#1}}

\newcommand{\norm}[2]{\ensuremath{\Vert #2 \Vert_{#1}}}
\newcommand{\abs}[1]{\ensuremath{\left\vert#1\right\vert}}
\newcommand{\inner}[2]{\ensuremath{\left\langle #2 \right\rangle}_{#1}}

\newcommand{\eps}{\varepsilon}

\newcommand{\RWup}[1]{\ensuremath{P^{\wedge}_{#1}}}
\newcommand{\Pup}[1]{\ensuremath{P^{\uparrow}_{#1}}}
\newcommand{\RWdown}[1]{\ensuremath{P^{\vee}_{#1}}}
\newcommand{\Pdown}[1]{\ensuremath{P^{\downarrow}_{#1}}}
\newcommand{\RWBS}{\ensuremath{P_{\textnormal{BX},\pi}}}

\newcommand{\Diri}[2]{\ensuremath{\+E_{#1}\left(#2\right)}}
\newcommand{\Ent}[2]{\ensuremath{\textnormal{Ent}_{#1}\left(#2\right)}}
\newcommand{\Var}[2]{\ensuremath{\textnormal{Var}_{#1}\left(#2\right)}}
\newcommand{\mixingtime}[1]{\ensuremath{t_{\textnormal{mix}}(#1)}}
\newcommand{\pimin}{\ensuremath{\pi_{\textnormal{min}}}}

\newcommand{\diag}{\operatorname{diag}}

\newcommand{\transpose}[1]{\ensuremath{#1^{\textnormal{\texttt{T}}}}}

\newcommand{\defeq}{:=}
\newcommand{\identity}{\ensuremath{\*I}}

\newcommand{\SCP}{\textbf{SCP}}
\newcommand{\NR}{\textbf{NR}}
\newcommand{\SRP}{\textbf{SRP}}
\newcommand{\SLC}{\textbf{SLC}}
\newcommand{\NCD}{\textbf{NCD}}
\newcommand{\covers}{\vartriangleright}

\newtheorem{theorem}{Theorem}

\newtheorem{lemma}[theorem]{Lemma}

\newtheorem{proposition}[theorem]{Proposition}
\newtheorem{corollary}[theorem]{Corollary}

\newtheorem{definition}[theorem]{Definition}
\newtheorem*{remark}{Remark}

\crefname{theorem}{Theorem}{Theorems}
\crefname{observation}{Observation}{Observations}
\crefname{claim}{Claim}{Claims}
\crefname{condition}{Condition}{Conditions}
\crefname{algorithm}{Algorithm}{Algorithms}
\crefname{property}{Property}{Properties}
\crefname{example}{Example}{Examples}
\crefname{fact}{Fact}{Facts}
\crefname{lemma}{Lemma}{Lemmas}
\crefname{corollary}{Corollary}{Corollaries}
\crefname{definition}{Definition}{Definitions}
\crefname{remark}{Remark}{Remarks}
\crefname{proposition}{Proposition}{Propositions}
\crefname{equation}{equation}{equations}
\crefname{enumi}{Case}{Case}
\creflabelformat{enumi}{(#2#1#3)}

\makeatletter
\def\prob#1#2#3{\goodbreak\begin{list}{}{\labelwidth\z@ \itemindent-\leftmargin
      \itemsep\z@  \topsep6\p@\@plus6\p@
      \let\makelabel\descriptionlabel}
  \item[\textbf{Name}]#1
  \item[\textbf{Instance}]#2
  \item[\textbf{Output}]#3
  \end{list}}
\makeatother

\title{Modified log-Sobolev inequalities for strongly log-concave distributions}

\author{Mary Cryan}
\author{Heng Guo}
\author{Giorgos Mousa}

\address{School of Informatics, University of Edinburgh, Informatics Forum, Edinburgh, EH8 9AB, United Kingdom.}
\email{mcryan@inf.ed.ac.uk, hguo@inf.ed.ac.uk, g.mousa@ed.ac.uk}

\begin{document}

\maketitle

\begin{abstract}
  We show that the modified log-Sobolev constant for a natural Markov chain 
  which converges to an $r$-homogeneous strongly log-concave distribution is at least $1/r$.
  Applications include a sharp mixing time bound for the bases-exchange walk for matroids,
  and a concentration bound for Lipschitz functions over these distributions.
\end{abstract}

\section{Introduction}

Let $\pi: 2^{[n]}\rightarrow \mathbb{R}_{\geq 0}$ be a discrete distribution, where $[n]=\{1,\ldots,n\}$.
Consider the generating polynomial of $\pi$:
\begin{align*}
    g_{\pi}(\*x)=\sum_{S \subseteq [n]}\pi(S)\prod_{i \in S}x_i.
\end{align*}
We call a polynomial \emph{log-concave} if its logarithm is concave,
and \emph{strongly log-concave} (\SLC) if it is log-concave at the all-ones vector $\*1$ after taking any sequence of partial derivatives.
The distribution $\pi$ is homogeneous and strongly log-concave if $g_{\pi}$ is.

An important example of homogeneous strongly log-concave distributions is the uniform distribution over the bases of a matroid \citep{ALOV19,BH19}.\footnote{For other examples, such as the determinantal point process and its variants, see \citep{ALOV19}.}
This discovery leads to the breakthrough result that the exchange walk over the bases of a matroid is rapidly mixing \citep{ALOV19},
which implies the existence of a fully polynomial-time randomised approximation scheme (FPRAS) for the number of bases of any matroid (given by an independence oracle).

The bases-exchange walk, denoted by $P_{\textnormal{BX}}$, is defined as follows.
In each step, we remove an element from the current basis uniformly at random to get a set $S$.
Then, we move to a basis containing $S$ uniformly at random.\footnote{Notice that to implement this step it may require more than constant time. The chain considered here is sometimes called the modified bases-exchange walk. A common alternative in the literature is to randomly propose an element and then apply a rejection filter. } %
This chain is irreducible and it converges to the uniform distribution over the bases of a matroid.
\cite{BH19} showed that the support of an $r$-homogeneous strongly log-concave distribution $\pi$ must be the set of bases of a matroid.
Thus, to sample from $\pi$, we may use a random walk $\RWBS$ similar to the above.
The only change required is that in the second step we move to a basis $B\supset S$ with probability proportional to $\pi(B)$.

Let $P$ be a Markov chain over a state space $\Omega$, and $\pi$ be its stationary distribution.
To measure the convergence rate of $P$, we use the total variation mixing time,
\begin{align*}
  \mixingtime{P,\eps}\defeq\min_t\left\{t\mid \| P^t(x_0,\cdot)-\pi\|_{\textnormal{TV}}\le \eps\right\},
\end{align*}
where $x_0\in\Omega$ is the initial state and the subscript $\textnormal{TV}$ denotes the total variation distance between two distributions.
The main goal of this paper is to show that for any $r$-homogeneous strongly log-concave distribution $\pi$,
\begin{align}\label{eqn:mixingtime-intro}
  \mixingtime{\RWBS,\eps}\le r\left( \log\log \frac{1}{\pimin}+\log\frac{1}{2\eps^2} \right),
\end{align}
where $\pimin=\min_{x\in\Omega} \pi(x)$.
This will improve the previous bound $\mixingtime{\RWBS,\eps} \le r\left( \log \frac{1}{\pimin}+\log\frac{1}{\eps} \right)$ due to \cite{ALOV19}.
Since $\pimin$ is most commonly exponentially small in the input size (e.g.\ when $\pi$ is the uniform distribution),
the improvement is usually a polynomial factor.
Our upper bound is sharp, 
as it is achieved (up to constant factors) when $\pi$ is the uniform distribution over the bases of some matroids \citep{Jer03}.\footnote{One such example is the matroid defined by a graph which is similar to a path but with two parallel edges connecting every two successive vertices instead of a single edge. Equivalently, this can be viewed as the partition matroid where each block has two elements and each basis is formed by choosing exactly one element from every block. The Markov chain $\RWBS$ in this case is just the 1/2-lazy random walk on the Boolean hypercube.}

Our main improvement is a modified log-Sobolev inequality (mLSI) for $\pi$ and $\RWBS$.
To introduce this inequality, we define the \emph{Dirichlet form} of a reversible Markov chain $P$, over state space $\Omega$, as
\begin{align*}
  \Diri{P}{f,g} \defeq  \sum_{x,y\in\Omega}\pi(x)f(x)\big[\identity-P\big](x,y)g(y),
\end{align*}
where $f,g$ are two functions over $\Omega$, and $\identity$ denotes the identity matrix.
Moreover, let the (normalised) relative entropy of $f:\Omega\rightarrow\=R_{\ge 0}$ be
\begin{align*}
    \Ent{\pi}{f}\defeq\Ex_{\pi}(f\log f)-\Ex_{\pi}f\log \Ex_{\pi}f,
\end{align*}
where we follow the convention that $0\log 0=0$.
If we normalise $\Ex_{\pi}f=1$, 
then $\Ent{\pi}{f}$ is the relative entropy (or Kullback--Leibler divergence) between $\pi(\cdot) f(\cdot)$ and $\pi(\cdot)$.

The modified log-Sobolev constant \citep{BT06} is defined as
\begin{align*}
  \rho(P)\defeq\inf\left\{\frac{\Diri{P}{f,\log f}}{\Ent{\pi}{f}}\mid\; f:\Omega\rightarrow \=R_{\ge 0}\;,\;\Ent{\pi}{f}\neq 0\right\}.
\end{align*}

Our main theorem is the following, which is a special case of \Cref{thm:main-complete}.

\begin{theorem}  \label{thm:main}
  Let $\pi$ be an $r$-homogeneous strongly log-concave distribution,
  and $\RWBS$ is the corresponding bases-exchange walk.
  Then
  \begin{align*}
    \rho(\RWBS)\ge \frac{1}{r}.
  \end{align*}
\end{theorem}

Since $\mixingtime{P,\eps}\le \frac{1}{\rho(P)}\left( \log\log \frac{1}{\pimin}+\log\frac{1}{2\eps^2} \right)$ \citep[cf.][]{BT06},
\Cref{thm:main} directly implies the mixing time bound \eqref{eqn:mixingtime-intro}.

In fact, we show more than \Cref{thm:main}.
Following \cite{ALOV19} and \cite{KO18},
we stratify independent sets of the matroid $\+M$ by their sizes,
and define two random walks for each level,
depending on whether they add or delete an element first.
For instance, the bases-exchange walk $\RWBS$ is the ``delete-add'' or ``down-up'' walk for the top level.
We give lower bounds for the modified log-Sobolev constants of both random walks for all levels.
For the complete statement, see \Cref{sec:main} and \Cref{thm:main-complete}.

The previous work of \cite{ALOV19}, building upon \citep{KO18}, focuses on the spectral gap of $\RWBS$.
It is well known that lower bounds of the modified log-Sobolev constant are stronger than those of the spectral gap.
Thus, we need to seek a different approach.
Our key lemma, \Cref{lem:entropy-contraction}, shows that the relative entropy decays by a factor of $1-\frac{1}{k}$ 
when we go from level $k$ to level $k-1$.
\Cref{thm:main} is a simple consequence of this lemma and Jensen's inequality.
In order to prove this lemma, we used a decomposition idea to inductively bound the relative entropy.
Although similar ideas have appeared before \citep{LY98,JSTV04,Mor09,Mor13},
our approach does not seem to fall into any existing framework.

Prior to our work, similar bounds have been obtained mostly for strong Rayleigh distributions,
which, introduced by \cite{BBL09}, are a proper subset of strongly log-concave distributions.
\cite{HS19} showed a lower bound on the modified log-Sobolev constant for strong Rayleigh distributions,\footnote{The result of \cite{HS19} in fact requires a weaker assumption, namely the \emph{stochastic covering property} (\SCP). We construct examples in \Cref{sec:SCP} to show that \SCP{} and \SLC{} are in fact incomparable.}
improving upon the spectral gap bound of \cite{AOR16}.
The work of \cite{HS19} builds upon the previous work of \cite{JSTV04} for balanced matroids \citep{FM92}.
All of these results follow an inductive framework inspired by \cite{LY98},
which is apparently difficult to carry out in the case of general matroids or strongly log-concave distributions.
Our analysis of the relative entropy took a different path from this line of work.

By the standard Herbst argument \citep[see, e.g.,][]{Goel04,Sam05,BLM13}, 
\Cref{thm:main} also implies the following concentration bound.
\begin{corollary}\label{cor:concentration}
  Let $\pi$ be an $r$-homogeneous strongly log-concave distribution with support $\Omega\subset 2^{[n]}$,
  and $\RWBS$ be the corresponding bases-exchange walk.
  For any observable function $f:\Omega\rightarrow\=R$ and $a\geq 0$,
  \begin{align*}
    \Pr_{x\sim \pi}(\abs{f(x)-\Ex_{\pi} f} \geq a) ~\leq~ 2\exp\left(-\frac{a^2}{2rv(f)}\right),
  \end{align*}
  where $v(f)$ is the maximum of one-step variances,
  \begin{align*}
    v(f) := \max_{x \in \Omega}\left\{\sum_{y \in \Omega}\RWBS(x,y)(f(x)-f(y))^2\right\}.
  \end{align*}
\end{corollary}

There have been a number of results concerning concentration inequalities for Lipshitz functions of negatively correlated variables.
\cite{pemantle_peres_2014} showed concentration for variables satisfying the \emph{stochastic covering property} (\SCP),
which includes strong Rayleigh distributions as special cases.
(See also \citealp{HS19}.)
Correcting an earlier proof of \cite{DR98},
\cite{GV18} showed concentration for variables with \emph{negative regression} (\NR), 
a property even weaker than \SCP.

For a $c$-Lipschitz function (under the graph distance in the bases-exchange graph), $v(f)\le c^2$.
Thus, \Cref{cor:concentration} generalises the concentration bound for Lipschitz functions in strong Rayleigh distributions.
However, \SLC{} is \emph{not} a negative correlation property.
We construct examples in \Cref{sec:SCP} to show that \SCP{} and \SLC{} are in fact incomparable.
Thus, \Cref{cor:concentration} is incomparable to the results of \cite{pemantle_peres_2014,HS19,GV18}.
It is not clear whether there is a larger class of distributions, generalising both \NR{} and \SLC, which retains this concentration bound. 

It is an interesting open problem to extend our result to more general settings.
\SLC{} distributions are special cases of high-dimensional expanders, where all local spectral gaps are at least $1$.
For more general cases, ``local-to-global'' bounds for spectral gaps have been obtained \citep{KO18,AL20}, whereas local-to-global mLSI on high-dimensional expanders is still elusive.
Another interesting setting is the uniform distribution over common bases or independent sets of two matroids.
Is there a Markov chain that converges rapidly to such distributions?
Note that this setting includes the important problem of sampling perfect matchings of bipartite graphs,
where the only known efficient algorithm is through an annealing process and its running time is a polynomial with high exponent \citep{JSV04}.

In \Cref{sec:prelim} we introduce necessary notions and briefly review relevant background.
In \Cref{sec:main} we formally state our main results.
In \Cref{sec:down-up} we show the decay of relative entropy and modified log-Sobolev constant lower bounds for the ``down-up'' and ``up-down'' walks.
In \Cref{sec:concentration} we show the concentration bound.
In \Cref{sec:SCP} we discuss stochastic covering property and strong log-concavity.

\section{Preliminaries}\label{sec:prelim}

In this section we define and give some basic properties of Markov chains, strongly log-concave distributions, and matroids.

\subsection{Markov chains}

Let $\Omega$ be a discrete state space and $\pi$ be a distribution over $\Omega$.
Let $P:\Omega\times\Omega\rightarrow \=R_{\ge0}$ be the transition matrix of a Markov chain whose stationary distribution is $\pi$.
Then, $\sum_{y\in\Omega}P(x,y)=1$ for any $x\in\Omega$.
We say $P$ is \emph{reversible} with respect to $\pi$ if 
\begin{align}\label{eqn:detailed-balance}
  \pi(x)P(x,y)=\pi(y)P(y,x).
\end{align}
We adopt the standard notation of $\Ex_{\pi}$ for a function $f:\Omega\rightarrow\=R$, namely
\begin{align*}
  \Ex_{\pi}f=\sum_{x\in\Omega}\pi(x)f(x).
\end{align*}

We also view the transition matrix $P$ as an operator that maps functions to functions.
More precisely, let $f$ be a function $f:\Omega\rightarrow \mathbb{R}$ and $P$ acting on $f$ is defined as
\begin{align*}
  Pf(x)\defeq\sum_{y\in\Omega}P(x,y)f(y).
\end{align*}
This is also called the \emph{Markov operator} corresponding to $P$.
We will not distinguish the matrix $P$ from the operator $P$ as it will be clear from the context.
Note that $Pf(x)$ is the expectation of $f$ with respect to the distribution $P(x,\cdot)$.
We can regard a function $f$ as a column vector in $\mathbb{R}^{\Omega}$,
in which case $Pf$ is simply matrix multiplication.

The Hilbert space $L_2(\pi)$ is given by endowing $\mathbb{R}^{\Omega}$ with the inner product
\begin{align}\notag
  \inner{\pi}{f,g}\defeq\sum_{x\in\Omega}\pi(x)f(x)g(x), 
\end{align}
where $f,g\in\mathbb{R}^{\Omega}$.
In particular, the norm in $L_2(\pi)$ is given by $\norm{\pi}{f} \defeq \sqrt{\inner{\pi}{f,f}}$.

The adjoint operator $P^*$ of $P$ is defined as $P^*(x,y)=\frac{\pi(y)P(y,x)}{\pi(x)}$.
This is the (unique) operator which satisfies $\inner{\pi}{f,Pg} = \inner{\pi}{P^*f,g}$.
It is easy to verify that if $P$ satisfies the detailed balanced condition \eqref{eqn:detailed-balance} (so $P$ is reversible),
then $P$ is \emph{self-adjoint}, namely $P=P^*$.

The \emph{Dirichlet form} is defined as:
\begin{align}  \label{eqn:Dirichlet}
  \Diri{P}{f,g}\defeq\inner{\pi}{(\identity-P)f,g},
\end{align}
where $\identity$ stands for the identity matrix of the appropriate size.
Let the \emph{Laplacian} $\+L\defeq \identity-P$.
Then, 
\begin{align}
  \Diri{P}{f,g} & = \sum_{x,y\in\Omega}\pi(x)g(x)\+L(x,y)f(y) \notag\\
  & = \transpose{g}\diag(\pi)\+Lf, \notag
\end{align}
where in the last line we regard $f$, $g$, and $\pi$ as (column) vectors over $\Omega$.
In particular, if $P$ is reversible, then $\+L^*=\+L$ and 
\begin{align}
  \Diri{P}{f,g} & = \inner{\pi}{\+Lf,g} =  \inner{\pi}{f,\+L^*g} = \inner{\pi}{f,\+Lg} = \Diri{P}{g,f} \notag\\
  & = \transpose{f}\diag(\pi)\+Lg\label{eqn:Dirichlet-alt}.
\end{align}
In this paper all Markov chains are reversible and we will most commonly use the form \eqref{eqn:Dirichlet-alt}.
Another common expression of the Dirichlet form for reversible $P$ is
\begin{align}\label{eqn:Dirichlet-reversible}
  \Diri{P}{f,g} & = \frac{1}{2}\sum_{x\in\Omega}\sum_{y\in\Omega}\pi(x)P(x,y)(f(x)-f(y))(g(x)-g(y)),
\end{align}
but we will not need this expression until \Cref{sec:concentration}.
It is well known that the spectral gap of $P$, or equivalently the smallest positive eigenvalue of $\+L$,
controls the convergence rate of $P$.
It also has a variational characterisation.
Let the variance of $f$ be
\begin{align}\notag
  \Var{\pi}{f}\defeq\Ex_{\pi}f^2-\left(\Ex_{\pi} f\right)^2.
\end{align}
Then
\begin{align}\notag
  \lambda(P)\defeq\inf\left\{\frac{\Diri{P}{f,f}}{\Var{\pi}{f}}\mid\; f:\Omega\rightarrow \=R\;,\;\Var{\pi}{f}\neq 0\right\}.
\end{align}
The usefulness of $\lambda(P)$ is due to the fact that, if, say, all eigenvalues of $P$ are non-negative, then
\begin{align}\label{eqn:Poincare-mixing}
  \mixingtime{P,\eps}\le \frac{1}{\lambda(P)}\left(\frac{1}{2} \log\frac{1}{\pimin}+\log\frac{1}{2\eps} \right),
\end{align}
where $\pimin=\min_{x\in\Omega}\pi(x)$.
See, for example, \citet[Theorem 12.4]{LP17}.

The (standard) log-Sobolev inequality relates $\Diri{P}{\sqrt{f},\sqrt{f}}$ with the following entropy-like quantity:
\begin{align} \label{eqn:entropy}
  \Ent{\pi}{f}\defeq\Ex_{\pi}(f\log f)-\Ex_{\pi}f\log \Ex_{\pi}f,
\end{align}
for a non-negative function $f$, where we follow the convention that $0\log 0=0$.
Also, $\log$ always stands for the natural logarithm in this paper.
The log-Sobolev constant is defined as 
\begin{align}\notag
  \alpha(P)\defeq\inf\left\{\frac{\Diri{P}{\sqrt{f},\sqrt{f}}}{\Ent{\pi}{f}}\mid\; f:\Omega\rightarrow \=R_{\ge 0}\;,\;\Ent{\pi}{f}\neq 0\right\}.
\end{align}
The constant $\alpha(P)$ gives a better control of the mixing time of $P$,
as shown by \cite{DS96},
\begin{align}\label{eqn:log-Sob-mixing}
  \mixingtime{P,\eps}\le \frac{1}{4\alpha(P)}\left(\log\log\frac{1}{\pimin}+\log\frac{1}{2\eps^2} \right).
\end{align}
The saving seems modest comparing to \eqref{eqn:Poincare-mixing},
but it is quite common that $\pimin$ is exponentially small in the instance size,
in which case the saving is a polynomial factor.

What we are interested in, however, is the following modified log-Sobolev constant introduced by \cite{BT06}:
\begin{align}\notag
  \rho(P)\defeq\inf\left\{\frac{\Diri{P}{f,\log f}}{\Ent{\pi}{f}}\mid\; f:\Omega\rightarrow \=R_{\ge 0}\;,\;\Ent{\pi}{f}\neq 0\right\}.
\end{align}
Similar to \eqref{eqn:log-Sob-mixing}, we have that
\begin{align}\label{eqn:mod-log-Sob-mixing}
  \mixingtime{P,\eps}\le \frac{1}{\rho(P)}\left(\log\log\frac{1}{\pimin}+\log\frac{1}{2\eps^2} \right),
\end{align}
as shown by \citet[Corollary 2.8]{BT06}. 

For reversible $P$, the following relationships among these constants are known,
\begin{align*}
  2\lambda(P)\ge\rho(P)\ge4\alpha(P).
\end{align*}
See, for example, \citet[Proposition 3.6]{BT06}.

Thus, lower bounds on these constants are increasingly difficult to obtain.
However, to get the best asymptotic control of the mixing time, 
one only needs to lower bound the modified log-Sobolev constant $\rho(P)$ instead of $\alpha(P)$
by comparing \eqref{eqn:log-Sob-mixing} and \eqref{eqn:mod-log-Sob-mixing}.
Indeed, as observed by \cite{HS19}, by taking the indicator function $\frac{1}{\pi(x)}\mathbbm{1}_x$ for all $x\in\Omega$,
\begin{align*}
  \alpha(P)\le\min_{x\in\Omega}\left\{ \frac{1}{-\log\pi(x)}\right\}.
\end{align*}
In our setting of $r$-homogeneous strongly log-concave distributions, 
we cannot hope for a uniform bound for $\alpha(P)$ similar to \Cref{thm:main},
as the right hand side of the above can be arbitrarily small for fixed $r$.

By \eqref{eqn:Dirichlet} and \eqref{eqn:entropy},
it is clear that if we replace $f$ by $cf$ for some constant $c>0$,
then both $\Diri{P}{f,\log f}$ and $\Ent{\pi}{f}$ increase by the same factor $c$.
Thus, in order to bound $\rho$, we may further assume that $\Ex_{\pi}f=1$. 
This assumption allows the simplification $\Ent{\pi}{f}=\Ex_{\pi}(f\log f)$.
In this case, $\pi(\cdot) f(\cdot)$ is a distribution,
and $\Ent{\pi}{f}$ is the relative entropy (or Kullback--Leibler divergence) between $\pi(\cdot) f(\cdot)$ and $\pi(\cdot)$.

\subsection{Strongly log-concave distributions}

We write $\partial_i$ as shorthand for $\frac{\partial}{\partial x_i}$,
and $\partial_I$ for an index set $I=\{i_1,\dots,i_k\}$ as shorthand for  $\partial_{i_1}\dots\partial_{i_k}$.

\begin{definition}\label{def:str-log-concave}
  A polynomial $p \in \=R[x_1,\dots,x_n]$ with non-negative coefficients is \emph{log-concave} at $\*x\in\=R_{\ge0}$ if its Hessian $\nabla^2\log p$ is negative semi-definite at $\*x$.
  We call $p$ \emph{strongly log-concave} if for any index set $I\subseteq[n]$, $\partial_I p$ is log-concave at the all-$1$ vector $\*1$.
\end{definition}

The notion of strong log-concavity was introduced by \cite{Gur09a,Gur09b}.
There are also notions of \emph{complete log-concavity} introduced by \cite{AOV18},
and \emph{Lorentzian} polynomials introduced by \cite{BH19}.
It turns out that for homogeneous polynomials the three notions are equivalent \citep[Theorem 5.3]{BH19}. 
(See also \citealp{ALOV19}.)

The following property of strongly log-concave polynomials is particularly useful \citep{AOV18,BH19}.
\begin{proposition}\label{prop:one-positive-eigenvalue}
  If $p$ is strongly log-concave,
  then for any $I\subseteq[n]$, the Hessian matrix $\nabla^2 \partial_I p(\*1)$ has at most one positive eigenvalue.
\end{proposition}
In fact, when p is homogeneous, $\nabla^2 \partial_I p(\*1)$ having at most one positive eigenvalue is equivalent to $\nabla^2\log \partial_I p(\*1)$ being negative semi-definite \citep{AOV18}, but we will only need the proposition above.

A distribution $\pi$ is called \emph{$r$-homogeneous} (or \emph{strongly log-concave}) if $g_{\pi}$ is.

\subsection{Matroids}

A matroid is a combinatorial structure that abstracts the notion of linear independence.
We shall define it in terms of its independent sets, although many different equivalent definitions exist.
Formally, a matroid $\+M=(E,\+I)$ consists of a finite ground set $E$ and a collection $\+I$ of subsets of $E$ (independent sets) that satisfy the following:
\begin{itemize}
  \item $\emptyset\in\+I$;
  \item if $S\in\+I$, $T\subseteq S$, then $T\in\+I$;
  \item if $S,T\in\+I$ and $\abs{S}>\abs{T}$, then there exists an element $i\in S\setminus T$ such that $T\cup\{i\}\in\+I$.
\end{itemize}
The first condition guarantees that $\+I$ is non-empty, the second implies that $\+I$ is downward closed, and the third is usually called the augmentation axiom.
We direct the reader to \cite{Oxl92} for a reference book on matroid theory.
In particular, the augmentation axiom implies that all the maximal independent sets have the same cardinality, namely the rank $r$ of $\+M$.
The set of bases $\+B$ is the collection of maximal independent sets of $\+M$.
Furthermore, we denote by $\+M(k)$ the collection of independent sets of size $k$, where $0\le k\le r$.
If we dropped the augmentation axiom, the resulting structure would be a non-empty collection of subsets of $E$ that is downward closed, known as an (abstract) \emph{simplicial complex}.

\citet[Theorem 7.1]{BH19} showed that the support of an $r$-homogeneous strongly log-concave distribution $\pi$ is the set of bases of a matroid $\+M=(E,\+I)$ of rank $r$.
We equip $\+I$ with a weight function $w(\cdot)$ recursively defined as follows:\footnote{One may define $w(I)$ to be a $\frac{k!}{r!}$ fraction of the current definition for $I\in\+M(k)$. This alternative definition will eliminate many factorial factors in the rest of the paper. However, it is inconsistent with the literature \citep{ALOV19,KO18}, so we do not adopt it.}
\begin{align*}
  w(I)\defeq 
  \begin{cases}
    \pi(I)Z_r & \text{if $\abs{I}=r$},\\
    \sum_{I'\supset I,\; \abs{I'}=\abs{I}+1} w(I') & \text{if $\abs{I}<r$},
  \end{cases}  
\end{align*}
for some normalisation constant $Z_r>0$.
For example, we may choose $w(B)=1$ for all $B\in\+B$ and $Z_r=\abs{\+B}$,
which corresponds to the uniform distribution over $\+B$.
It follows that 
\begin{align}\notag
  w(I) = (r-\abs{I})!\sum_{B\in\+B,\;I\subseteq B}w(B).
\end{align}
Let $\pi_k$ be the distribution over $\+M(k)$ such that $\pi_k(I)\propto w(I)$ for $I\in\+M(k)$.
Thus $\pi=\pi_r$.
For any $I\in\+M(k)$, 
$\pi_k(I)$ is proportional to the probability of generating a superset of~$I$ under $\pi$.
Let $Z_k=\sum_{I\in\+M(k)}w(I)$ be the normalisation constant of $\pi_k$.
In fact, for any $0\le k\le r$, $k!Z_k=Z_0=w(\emptyset)$.

It is straightforward to verify that for any $I\in\+I$,
\begin{align}\label{eqn:partial-I}
  \partial_Ig_{\pi}(\*1) = \sum_{B\in\+B,I\subset B}\pi(B) = \frac{1}{Z_r}\sum_{B\in\+B,I\subset B}w(B).
\end{align}
We also write $w(v)$ as shorthand for $w(\{v\})$ for any $v\in E$.

For an independent set $I\in\+I$, the contraction $\+M_I=(E\setminus I, \+I_I)$ is also a matroid, 
where $\+I_I=\{J\mid J\subseteq E\setminus I, J\cup I\in\+I\}$.
We equip $\+M_I$ with a weight function $w_I(\cdot)$ such that $w_I(J)=w(I\cup J)$.
We may similarly define distributions $\pi_{I,k}$ for $k\le r-\abs{I}$ such that $\pi_{I,k}(J)\propto w_I(J)$ for $J\in\+M_I(k)$.
For convenience, instead of defining $\pi_{I,k}$ over $\+M_I(k)$,
we define it over $\+M(k+\abs{I})$ such that for any $J\in\+M(k+\abs{I})$,
\begin{align} \label{eqn:pi-I-k}
  \pi_{I,k}(J)\defeq
  \begin{cases}
    \frac{k!w(J)}{w(I)} & \text{if $I\subset J$;}\\
    0 & \text{otherwise.}
  \end{cases}
\end{align}
Notice that the normalising constant $Z_{I,k}=\frac{w(I)}{k!}$.

If $\abs{I}\le r-2$, let $W_I$ be the matrix such that $W_{uv}=w_I(\{u,v\})$ for any $u,v\in E\setminus I$.
Then, notice that
\begin{align*}
  w_I(\{u,v\}) & = w(I\cup\{u,v\}) \\
  & = (r-\abs{I}-2)!\sum_{B\in\+B,\;I\cup\{u,v\}\subseteq B}w(B) \\
  & = (r-\abs{I}-2)!Z_{r}\cdot \partial_u\partial_v\partial_I g_{\pi} (\*1). \tag{by \eqref{eqn:partial-I}}
\end{align*}
In other words, $W_I$ is $\nabla^2\partial_I g_{\pi}$ multiplied by the scalar $(r-\abs{I}-2)!Z_{r}$.
Thus, \Cref{prop:one-positive-eigenvalue} implies the following.

\begin{proposition}\label{prop:one-eigenvalue}
  Let $\pi$ be an $r$-homogeneous strongly log-concave distribution over $\+M=(E,\+I)$.
  If $I\in\+I$ and $\abs{I}\le r-2$,
  then the matrix $W_I$ has at most one positive eigenvalue.
\end{proposition}

\Cref{prop:one-eigenvalue} implies the following bound for a quadratic form,
which will be useful later.

\begin{lemma}  \label{lem:quadratic-bound}
  Let $\pi$ be an $r$-homogeneous strongly log-concave distribution over $\+M=(E,\+I)$,
  and let $I\in\+I$ such that $\abs{I}\le r-2$.  Let $f:\+M_I(1)\rightarrow\=R$ be a function 
  such that $\Ex_{\pi_{I,1}}f=1$.  Then
  \begin{align*}
    \transpose{f}W_If\le w(I).
  \end{align*}
\end{lemma}
\begin{proof}
  Let $\*w_I=\{w_I(v)\}_{v\in E\setminus I}$.
  The constraint $\Ex_{\pi_{I,1}}f=1$ implies that $\sum_{v\in E\setminus I}w_I(v)f(v)=w(I)$.
  Let $D=\diag(\*w_I)$ and $A=D^{-1/2}W_ID^{-1/2}$.
  Then $A$ is a real symmetric matrix.
  By \Cref{prop:one-eigenvalue},
  $W_I$ has at most one positive eigenvalue, and thus so does $A$ \citep[see, e.g.,][Lemma 2.4]{ALOV19}.
  We may decompose $A$ as 
  \begin{align}\label{eqn:decompose-A}
    A =\sum_{i=1}^{\abs{E\setminus I}} \lambda_i g_i\transpose{g_i},
  \end{align}
  where $\{g_i\}$ is an orthonormal basis and $\lambda_i\le 0$ for all $i\ge 2$.
  Moreover, notice that $\sqrt{\*w_I}$ is an eigenvector of $A$ with eigenvalue $1$.
  Thus, $\lambda_1=1$ and $g_1$ can be taken as $\sqrt{\pi_{I,1}}$.

  The decomposition \eqref{eqn:decompose-A} directly implies that
  \begin{align*}
    W = \sum_{i=1}^{\abs{E\setminus I}} \lambda_i h_i\transpose{h_i},
  \end{align*}
  where $h_i=D^{1/2}g_i$. 
  In particular, $h_1=\frac{1}{\sqrt{w(I)}}\*w_I$.

  The assumption $\sum_{v\in E\setminus I}w_I(v)f(v)=w(I)$ can be rewritten as 
  $\inner{}{h_1,f}=\sqrt{w(I)}$.
  Thus,
  \begin{align*}
    \transpose{f}W_If
    & = \sum_{i=1}^{\abs{E\setminus I}} \lambda_i \inner{}{h_i,f}^2 \le \inner{}{h_1,f}^2 = w(I),
  \end{align*}
  where the inequality is due to the fact that $\lambda_1=1$ and $\lambda_i\le 0$ for all $i\ge 2$.
  The lemma follows.
\end{proof}

\section{Main results}\label{sec:main}

There are two natural random walks $\RWup{k}$ and $\RWdown{k}$ on $\+M(k)$ by starting with adding or deleting an element and coming back to $\+M(k)$.
Given the current $I\in\+M(k)$, 
the ``up-down'' random walk $\RWup{k}$ first chooses $I'\in\+M(k+1)$ such that $I'\supset I$ with probability proportional to $w(I')$,
and then removes one element from $I'$ uniformly at random.
More formally, for $1\le k\le r-1$ and $I,J\in\+M(k)$, we have that
\begin{align}  \label{eqn:RWup}
  \RWup{k}(I,J)=  
  \begin{cases}
    \frac{1}{k+1} & \text{if $I=J$};\\
    \frac{w(I\cup J)}{(k+1)w(I)} & \text{if $I\cup J\in\+M(k+1)$};\\
    0 & \text{otherwise.}
  \end{cases}
\end{align}
The ``down-up'' random walk $\RWdown{k}$ removes an element of $I$ uniformly at random to get $I'\in\+M(k-1)$,
and then moves to $J$ such that $J\in\+M(k),J\supset I'$ with probability proportional to $w(J)$.
More formally, for $2\le k\le r$,
\begin{align}  \label{eqn:RWdown}
  \RWdown{k}(I,J)=  
  \begin{cases}
    \sum_{I'\in\+M(k-1),I'\subset I}\frac{w(I)}{kw(I')} & \text{if $I=J$};\\
    \frac{w(J)}{kw(I\cap J)} & \text{if $\abs{I\cap J}=k-1$};\\
    0 & \text{if $\abs{I\cap J}<k-1$}.
  \end{cases}
\end{align}
Thus, the bases-exchange walk $\RWBS$ according to $\pi$ is just $\RWdown{r}$.
The stationary distribution of both $\RWup{k}$ and $\RWdown{k}$ is $\pi_k(I) = \frac{w(I)}{Z_k}=\frac{k! w(I)}{r!Z_r}$.

\begin{theorem}  \label{thm:main-complete}
  Let $\pi$ be an $r$-homogeneous strongly log-concave distribution,
  and $\+M$ the associated matroid.
  Let $\RWdown{k}$ and $\RWup{k}$ be defined as above on $\+M(k)$.
  Then the following hold:
  \begin{itemize}
    \item for any $2\le k\le r$, $\rho(\RWdown{k})\ge\frac{1}{k};$
    \item for any $1\le k\le r-1$, $\rho(\RWup{k})\ge\frac{1}{k+1}$.
  \end{itemize}
\end{theorem}

\Cref{thm:main-complete} is shown in \Cref{sec:down-up}.
Interestingly, we do not know how to directly relate $\rho(\RWup{k})$ with $\rho(\RWdown{k+1})$,
although it is straightforward to see that both walks have the same spectral gap (see \eqref{eqn:up-decomp} and \eqref{eqn:down-decomp} below).

By \eqref{eqn:mod-log-Sob-mixing}, we have the following corollary.
\begin{corollary}\label{cor:mixing}
  In the same setting as \Cref{thm:main-complete},
  we have that
  \begin{itemize}
    \item for any $2\le k\le r$, $\mixingtime{\RWdown{k},\eps}\le k\left( \log\log\pi_{k,\textnormal{min}}^{-1}+\log \frac{1}{2\eps^2} \right)$;
    \item for any $1\le k\le r-1$, $\mixingtime{\RWup{k},\eps}\le (k+1)\left( \log\log\pi_{k,\textnormal{min}}^{-1}+\log \frac{1}{2\eps^2} \right)$.
  \end{itemize}
  In particular,
  for the bases-exchange walk $\RWBS$ according to $\pi$,
  \begin{align*}
    \mixingtime{\RWBS,\eps}\le r\left( \log\log\pi_{\textnormal{min}}^{-1}+\log \frac{1}{2\eps^2} \right).
  \end{align*}
\end{corollary}

Let $\+M$ be a matroid of rank $r$ with a ground set of size $n$.
For the uniform distribution over the bases of $\+M$,
\Cref{cor:mixing} implies that the mixing time of the bases-exchange walk is $O(r(\log r+\log\log n))$,
which improves upon the $O(r^2\log n)$ bound of \cite{ALOV19}.
The mixing time bound in \Cref{cor:mixing} is sharp, 
as there are matroids where the upper bound is achieved \citep[Ex.\ 9.14]{Jer03}.
As mentioned in the introduction, one such example is the graphic matroid defined by a graph 
which is similar to a path but with two parallel edges connecting every two successive vertices instead of a single edge. 
Equivalently, this can be viewed as the partition matroid where each block has two elements 
and each basis is formed by choosing exactly one element from every block. 
The rank of this matroid is $r=n/2$, and $\pimin=\frac{1}{2^{n/2}}$.
The Markov chain $\RWBS$ in this case is just the 1/2-lazy random walk on the $n/2$-dimensional Boolean hypercube,
which has mixing time $\Theta(n\log n)$, matching the upper bound in \Cref{cor:mixing}.

For more details on the concentration result, \Cref{cor:concentration}, see \Cref{sec:concentration}.

\section{Decay of relative entropy}\label{sec:down-up}

In this section and what follows,
we always assume that the matroid $\+M$ and the weight function $w(\cdot)$ correspond to an $r$-homogeneous strongly log-concave distribution $\pi=\pi_r$.

We first give some basic decompositions of $\RWdown{k}$ and $\RWup{k}$.
Let $A_k$ be a matrix whose rows are indexed by $\+M(k)$ and columns by $\+M(k+1)$ such that
\begin{align*}
  A_k(I,J) \defeq 
  \begin{cases}
    1 & \text{if $I\subset J$};\\
    0 & \text{otherwise},
  \end{cases}
\end{align*}
and $\*w_k$ be the vector of $\{w(I)\}_{I\in\+M(k)}$.
Moreover, let
\begin{align}
  \Pup{k} &\defeq \diag(\*w_k)^{-1}A_k\diag(\*w_{k+1}), \label{eqn:Pup}\\
  \Pdown{k+1} &\defeq \frac{1}{k+1}\transpose{A_k}. \label{eqn:Pdown}
\end{align}
Then
\begin{align}
  \RWup{k} & = \Pup{k}\Pdown{k+1}, \label{eqn:up-decomp}\\
  \RWdown{k+1}& = \Pdown{k+1}\Pup{k}.   \label{eqn:down-decomp}
\end{align}
Let $D_k=\diag(\pi_k)$.
Using \eqref{eqn:Pup} and \eqref{eqn:Pdown},
we get that
\begin{align}
  D_{k+1}\Pdown{k+1} = \transpose{(\Pup{k})}D_k.
  \label{eqn:Pup-Pdown}
\end{align}
By multiplying equation \eqref{eqn:Pup-Pdown} by the all-ones vector, we also get that
\begin{align}
  \pi_{k+1}\Pdown{k+1} & =\pi_{k}, \label{eqn:pi-Pdown}\\
  \pi_{k}\Pup{k} & =\pi_{k+1}. \label{eqn:pi-Pup}
\end{align}
For $k\ge 2$ and a function $f^{(k)}:\+M(k)\rightarrow\=R_{\ge 0}$,
define $f^{(i)}:\+M(i)\rightarrow\=R_{\ge 0}$ for $1\le i\le k-1$ such that
\begin{align}  \label{eqn:f-level-i}
  f^{(i)}\defeq\prod_{j=i}^{k-1}\Pup{j} f^{(k)}.
\end{align}
Intuitively, $f^{(i)}$ is the function $f^{(k)}$ ``going down'' to level $i$.
The key lemma, namely \Cref{lem:entropy-contraction}, is that this operation contracts the relative entropy by a factor of $1-\frac{1}{i}$ from level $i$ to level $i-1$.

In fact, recall that if we normalise $\Ex_{\pi_k}f^{(k)}=1$, then $\transpose{\left( f^{(k)} \right)} D_{k}$ is a distribution (viewed as a row vector).
Then, it is easy to verify that 
\begin{align}\label{eqn:left-multiply}
  \transpose{\left( f^{(k-1)} \right)} D_{k-1} = \transpose{\left( f^{(k)} \right)} D_{k} \Pdown{k}.
\end{align}
Namely, the corresponding distribution of $f^{(k-1)}$ is that of $f^{(k)}$ after the random walk $\Pdown{k}$.


We first establish some properties of $f^{(i)}$ for $i<k$.

\begin{lemma}\label{lem:f-k}
  Let $k\ge 2$ and $f^{(k)}:\+M(k)\rightarrow\=R_{\ge 0}$ be a non-negative function on $\+M(k)$.
  Then we have the following:
  \begin{enumerate}
    \item \label{case:f-sup-i} for any $1\le i<k$, $J\in\+M(i)$, $f^{(i)}(J)=\Ex_{\pi_{J,k-i}}f^{(k)};$
    \item \label{case:normalisation-level-i} for any $1\le i\le k$, $\Ex_{\pi_{i}}f^{(i)}=\Ex_{\pi_{k}}f^{(k)}$.
  \end{enumerate}
\end{lemma}
\begin{proof}
  For \eqref{case:f-sup-i}, first notice that
  \begin{align*}
      \transpose{\delta_{J}}\prod_{j=i}^{k-1}\Pup{j}=\transpose{\delta_{J}}\prod_{j=i}^{k-1}\left[\diag(\*w_j)^{-1}A_j\diag(\*w_{j+1})\right]=\frac{\transpose{\delta_{J}}}{w(J)}\prod_{j=i}^{k-1}A_j\diag(\*w_{k})=\pi_{J,k-i},
  \end{align*}
  where $\delta_{J}$ is the Dirac vector that equals $1$ at $J$ and $0$ elsewhere. The last equality holds due to the fact that the product of the adjacency matrices counts the paths from independent sets at level $i$ to independent sets at level $k$. For every such pair of sets, the number of these paths is $(k-i)!$ if one is contained in the other, or $0$ otherwise. It follows that
  \begin{align*}
    \Ex_{\pi_{J,k-i}}f^{(k)} = \pi_{J,k-i}f^{(k)} = \transpose{\delta_{J}}\prod_{j=i}^{k-1}\Pup{j}f^{(k)} = \transpose{\delta_{J}}f^{(i)} = f^{(i)}(J).
  \end{align*}
  For \eqref{case:normalisation-level-i}, we have that
  \begin{align*}
    \Ex_{\pi_{i}}f^{(i)} & = \pi_{i}\prod_{j=i}^{k-1}\Pup{j} f^{(k)} \\
    & = \pi_{k}f^{(k)} \tag{by \Cref{eqn:pi-Pup}}\\
    & =  \Ex_{\pi_{k}}f^{(k)}. \qedhere
  \end{align*}
\end{proof}
Now we are ready to establish the base case of the entropy's contraction.

\begin{lemma}  \label{lem:k=2}
  Let $f^{(2)}:\+M(2)\rightarrow\=R_{\ge 0}$ be a non-negative function defined on $\+M(2)$.
  Then,
  \begin{align*}
    \Ent{\pi_2}{f^{(2)}} \ge 2\Ent{\pi_1}{f^{(1)}}.
  \end{align*}
\end{lemma}
\begin{proof}
  Without loss of generality we may assume that $\Ex_{\pi_2}f^{(2)}=1$ and therefore $\Ex_{\pi_1}f^{(1)}=1$ by \eqref{case:normalisation-level-i} of \Cref{lem:f-k}.
  Note that for $v\in E$,
  \begin{align*}
    f^{(1)}(v) = \sum_{S\in\+M(2):v\in S}\frac{w(S)}{w(v)}f^{(2)}(S).
  \end{align*}
  We will use the following inequality, which is valid for any $a\ge 0$ and $b > 0$,
  \begin{align}\label{eqn:log-a-b}
    a\log\frac{a}{b}\ge a-b.
  \end{align}
  Noticing that $Z_1=2Z_2$, we have
  \begin{align*}
    &\Ent{\pi_2}{f^{(2)}} - 2 \Ent{\pi_{1}}{f^{(1)}} \\
    =\; & \sum_{S\in\+M(2)}\pi_2(S)f^{(2)}(S)\log f^{(2)}(S)-2 \sum_{v\in E}\pi_1(v)\left(\sum_{S\in\+M(2):v\in S}\frac{w(S)}{w(v)}f^{(2)}(S)\right)\log f^{(1)}(v)\\
    =\; & \sum_{S\in\+M(2)}\left(\pi_2(S)f^{(2)}(S)\log f^{(2)}(S) - 
    2\sum_{v\in S}\pi_1(v)\frac{w(S)}{w(v)}f^{(2)}(S)\log f^{(1)}(v)\right)\\
    =\; & \sum_{S\in\+M(2)}\left(\frac{w(S)}{Z_2}f^{(2)}(S)\log f^{(2)}(S) - 
    2\sum_{v\in S}\frac{w(v)}{Z_1}\cdot\frac{w(S)}{w(v)}f^{(2)}(S)\log f^{(1)}(v)\right)\\
    =\; & \sum_{S=\{u,v\}\in\+M(2)}\frac{w(S)}{Z_2}f^{(2)}(S)\left(\log f^{(2)}(S) - \log f^{(1)}(v)-\log f^{(1)}(u)\right)\\
    \ge\; & \sum_{S=\{u,v\}\in\+M(2)}\frac{w(S)}{Z_2}\left(f^{(2)}(S) - f^{(1)}(v)f^{(1)}(u)\right)\\
    =\; & \sum_{S\in\+M(2)}\pi_2(S)f^{(2)}(S) - \sum_{S=\{u,v\}\in\+M(2)}\frac{w(S)}{Z_2}\cdot f^{(1)}(v)f^{(1)}(u)\\
    =\; & 1 - \frac{1}{2Z_2}\cdot\transpose{\left( f^{(1)} \right)}W_{\emptyset}f^{(1)},
  \end{align*}
  where the inequality is by \eqref{eqn:log-a-b} with $a=f^{(2)}(S)$ and $b=f^{(1)}(u)f^{(1)}(v)$ when $b>0$,
  and when $b=0$ we have $a=0$ as well.
  Thus, the lemma follows from \Cref{lem:quadratic-bound} with $I=\emptyset$ and $w(\emptyset)=Z_1=2Z_2$.
\end{proof}

We generalise \Cref{lem:k=2} as follows.

\begin{lemma}  \label{lem:entropy-contraction}
  Let $k\ge 2$ and $f^{(k)}:\+M(k)\rightarrow\=R_{\ge 0}$ be a non-negative function defined on $\+M(k)$.
  Then 
  \begin{align*}
    \Ent{\pi_k}{f^{(k)}} \ge \frac{k}{k-1}\Ent{\pi_{k-1}}{f^{(k-1)}}.
  \end{align*}
\end{lemma}
\begin{proof}
  We do an induction on $k$.
  The base case of $k=2$ follows from \Cref{lem:k=2}.

  For the induction step, assume the lemma holds for all integers at most $k$ for any matroid $\+M$.
  Let $f^{(k+1)}:\+M(k+1)\rightarrow\=R_{\ge 0}$ be a non-negative function such that $\Ex_{\pi_{k+1}}f^{(k+1)}=1$.

  Recall \eqref{eqn:pi-I-k},
  where we define $\pi_{v,k}$ over $\+M(k+1)$ instead of over $\+M_v(k)$.
  For $I\in\+M(k+1)$, $v\in\+M(1)$ and $v\in I$,
  \begin{align*}
    \pi_{k+1}(I) = \frac{w(I)}{Z_{k+1}} = (k+1)\cdot\frac{w(v)}{(k+1)!Z_{k+1}}\cdot \frac{k!w(I)}{w(v)} = (k+1)\pi_{1}(v)\pi_{v,k}(I),
  \end{align*}
  as $Z_{1}=(k+1)!Z_{k+1}$.
  This means that
  \begin{align}\label{eqn:decomp}
    \pi_{k+1}(I) =  \sum_{v\in\+M(1),v\in I}\pi_{1}(v)\pi_{v,k}(I) = \sum_{v\in\+M(1)}\pi_{1}(v)\pi_{v,k}(I).
  \end{align}
  Thus $\pi_{k+1}$ is a mixture of $\pi_{v,k}$.

  We use the ``chain rule'' of entropy to 
  decompose $\Ent{\pi_{k+1}}{f^{(k+1)}}$ with respect to the entropy of $f^{(1)}$ 
  (``projection'') and the entropy conditioned on having each $v$ (``restriction'').
  To be more precise, we have
  \begin{align*}
    \Ex_{\pi_{k+1}}f^{(k+1)}\log f^{(k+1)} = \sum_{v\in\+M(1)}\pi_{1}(v) \Ex_{\pi_{v,k}}f^{(k+1)}\log f^{(k+1)}.
  \end{align*}
  This implies that
  \begin{align}
    & \hspace{0.7cm} \Ent{\pi_{k+1}}{f^{(k+1)}} \notag\\
    & = \sum_{v\in\+M(1)}\pi_{1}(v)\Ent{\pi_{v,k}}{f^{(k+1)}} +\sum_{v\in\+M(1)}\pi_1(v)\left(\Ex_{\pi_{v,k}}f^{(k+1)}\right)\log \left( \Ex_{\pi_{v,k}}f^{(k+1)} \right)\notag\\
    & = \sum_{v\in\+M(1)}\pi_{1}(v)\Ent{\pi_{v,k}}{f^{(k+1)}} + \Ent{\pi_1}{f^{(1)}},\label{eqn:decomp-k+1}
  \end{align}
  where we use \eqref{case:f-sup-i} and \eqref{case:normalisation-level-i} of \Cref{lem:f-k}.
  Similarly,
  \begin{align}    \label{eqn:decomp-k}
    \Ent{\pi_{k}}{f^{(k)}} 
    & = \sum_{v\in\+M(1)}\pi_{1}(v)\Ent{\pi_{v,k-1}}{f^{(k)}} + \Ent{\pi_1}{f^{(1)}}.
  \end{align}

  For any $v\in\+M(1)$, the contracted matroid $\+M_v$ with weight function $w_v(I)=w(I\cup v)$ for $I\subseteq E\setminus\{v\}$
  corresponds to an $(r-1)$-homogeneous strongly log-concave distribution. (Recall \Cref{def:str-log-concave}.)
  Thus, we can apply the induction hypothesis on $\+M_v$ at level $k$ and get
  \begin{align}\label{eqn:entropy-IH}
    \Ent{\pi_{v,k}}{f^{(k+1)}} \ge \frac{k}{k-1}\cdot\Ent{\pi_{v,k-1}}{f^{(k)}}.
  \end{align}
  Strictly speaking, in \eqref{eqn:entropy-IH} we should apply the induction hypothesis to $f_v^{(k)}$ which is the restriction of $f^{(k+1)}$ to $J\in\+M(k+1)$ and $J\ni v$,
  and then ``push it down'' to $f_v^{(k-1)}$ defined over $I\in\+M(k)$ and $I\ni v$ as
  \begin{align*}
    f_v^{(k-1)}(I)\defeq\sum_{J\in\+M(k+1):J\supset I}\frac{w(J)}{w(I)}\cdot f_v^{(k)}(J) = \sum_{J\in\+M(k+1):J\supset I}\frac{w(J)}{w(I)}\cdot f^{(k+1)}(J).
  \end{align*}
  However, $f_v^{(k)}$ agrees with $f^{(k+1)}$ on the support of $\pi_{v,k}$,
  and $f_v^{(k-1)}$ agrees with $f^{(k)}$ on the support of $\pi_{v,k-1}$.
  This validates \eqref{eqn:entropy-IH}.

  Furthermore, using the induction hypothesis on $\+M$ from level $k$ to level $1$, we have that
  \begin{align}    \label{eqn:entropy-IH-1-k}
    \Ent{\pi_{k}}{f^{(k)}}\ge k\cdot\Ent{\pi_{1}}{f^{(1)}}.
  \end{align}
  Thus, \eqref{eqn:decomp-k} and \eqref{eqn:entropy-IH-1-k} together imply that
  \begin{align}    \label{eqn:sum-k-level-1}
    \sum_{v\in\+M(1)}\pi_{1}(v)\Ent{\pi_{v,k-1}}{f^{(k)}} \ge (k-1)\Ent{\pi_1}{f^{(1)}}.
  \end{align}
  Putting everything together,
  \begin{align*}
    \Ent{\pi_{k+1}}{f^{(k+1)}} &= \sum_{v\in\+M(1)}\pi_{1}(v)\Ent{\pi_{v,k}}{f^{(k+1)}} + \Ent{\pi_1}{f^{(1)}} \tag{by \eqref{eqn:decomp-k+1}}\\
    & \ge \frac{k}{k-1}\sum_{v\in\+M(1)}\pi_{1}(v)\Ent{\pi_{v,k-1}}{f^{(k)}} + \Ent{\pi_1}{f^{(1)}} \tag{by \eqref{eqn:entropy-IH}}\\
    & = \left(\frac{k+1}{k}+\frac{1}{k(k-1)}\right)\sum_{v\in\+M(1)}\pi_{1}(v)\Ent{\pi_{v,k-1}}{f^{(k)}}+\Ent{\pi_1}{f^{(1)}}\\
    & \ge \frac{k+1}{k}\sum_{v\in\+M(1)}\pi_{1}(v)\Ent{\pi_{v,k-1}}{f^{(k)}} + \frac{k+1}{k}\Ent{\pi_1}{f^{(1)}} \tag{by \eqref{eqn:sum-k-level-1}}\\
    & = \frac{k+1}{k}\Ent{\pi_{k}}{f^{(k)}}.\tag{by \eqref{eqn:decomp-k}}
  \end{align*}
  This concludes the inductive step and thus the proof.
\end{proof}
\begin{remark}
  We remark that our decomposition of the relative entropy \eqref{eqn:decomp-k+1} is ``horizontal'' with respect to elements of $\+M(1)$.
  This decomposition is different from the decomposition by \citet[Theorem 5.2]{KO18} in a similar context,
  where they decompose ``vertically'' across all levels.
\end{remark}

The decomposition \eqref{eqn:decomp} of $\pi_k$ appears to be the key to \Cref{lem:entropy-contraction}.
An alternative way to understand it is the following.
Consider the process which first draws a basis $B\sim \pi$,
and then repeatedly removes an element from the current set uniformly 
at random for at most~$r$ repetitions.
Let $X_k$ be the outcome of this process after removing $r-k$ elements.
Then $\abs{X_k}=k$, and $\pi_k(I)=\Pr(X_k=I)$ for $I\in\+M(k)$.
Moreover, 
\begin{align*}
  \Pr(X_1=\{v\}\mid X_k=I)=
  \begin{cases}
    \frac{1}{k} & \text{if $v\in I$};\\
    0 & \text{otherwise}.
  \end{cases}
\end{align*}
By Bayes' rule,
\begin{align*}
  \Pr(X_k=I\mid X_1=\{v\})\Pr(X_1=\{v\}) = \Pr(X_1=\{v\}\mid X_k=I)\Pr(X_k=I).
\end{align*}
Summing over $v$, since $\sum_{v\in \+M(1)}\Pr(X_1=\{v\}\mid X_k=I)=1$, we have
\begin{align}
  &\sum_{v\in \+M(1)}\Pr(X_k=I\mid X_1=\{v\})\Pr(X_1=\{v\})\notag\\
  =~&\Pr(X_k=I)\sum_{v\in \+M(1)}\Pr(X_1=\{v\}\mid X_k=I)\notag\\
  =~&\Pr(X_k=I). \label{eqn:decomp-2}
\end{align}
Noticing that $\Pr(X_k=I\mid X_1=\{v\}) = \pi_{v,k-1}(I)$,
equation \eqref{eqn:decomp-2} recovers \eqref{eqn:decomp}.

By recalling \eqref{eqn:f-level-i} and \eqref{eqn:left-multiply},
we observe that the analysis of the ``going-down'' half---and, similarly, the ``going-up'' half---of $\RWdown{k}$ and $\RWup{k-1}$ corresponds to premultiplying by $\Pup{k-1}$---or, accordingly, $\Pdown{k}$---to a function $f$.
Hence, \Cref{lem:entropy-contraction} implies that the relative entropy contracts by $1-\frac{1}{k}$ in the ``going-down'' half of the random walks $\RWdown{k}$ and $\RWup{k-1}$.
What we show next is that the other half will not increase the relative entropy; a fact which is a special case of the so-called ``data processing inequality''.

\begin{lemma}  \label{lem:data-processing}
  For any $k\ge 2$ and $f:\+M(k-1)\rightarrow\=R_{\ge 0}$,
  \begin{align}\label{eqn:Pdown-entropy}
    \Ent{\pi_k}{\Pdown{k}f} \leq \Ent{\pi_{k-1}}{f}.
  \end{align}
\end{lemma}
\begin{proof}
  Firstly, we verify that
  \begin{align*}
    \Ex_{\pi_{k}}\Pdown{k}f & = \pi_k\Pdown{k}f\\
    &=\pi_{k-1}f =\Ex_{\pi_{k-1}}f. \tag{by \Cref{eqn:pi-Pdown}}
  \end{align*}
  Thus, we can assume both are $1$ without loss of generality. 
  Then,
  \begin{align*}
    \Ent{\pi_k}{\Pdown{k}f} & = \pi_k(\Pdown{k}f \odot \log \Pdown{k}f) \\
    & \leq  \pi_k\Pdown{k}(f \odot \log f) \tag{by Jensen's inequality on $x \log x$} \\
    & = \pi_{k-1}(f \odot \log f) \tag{by \Cref{eqn:pi-Pdown}}\\
    & = \Ent{\pi_{k-1}}{f},
  \end{align*}
  where $\odot$ stands for the Hadamard product.  
\end{proof}

With \Cref{lem:entropy-contraction,lem:data-processing} in hand, we can show the decay of relative entropy for $\RWdown{k}$ and $\RWup{k}$.

\begin{corollary}  \label{cor:rel-entropy-contraction}
  For any distribution $\tau$ on $\+M(k)$,
  \begin{itemize}
    \item if $2\le k\le r$, then $D\left(\tau\RWdown{k}~\|~\pi_k\right) \le \left(1-\frac{1}{k}\right) D(\tau~\|~\pi_k)$;
    \item if $1\le k\le r-1$, then $D\left(\tau\RWup{k}~\|~\pi_k\right) \le \left(1-\frac{1}{k+1}\right) D(\tau~\|~\pi_k)$.
  \end{itemize}
\end{corollary}
\begin{proof}
  We will only prove this corollary for $\RWdown{k}$ as the case of $\RWup{k}$ is similar. 
  We have that $D(\tau~\|~\pi_k) = \Ent{\pi_k}{D_k^{-1}\transpose{\tau}}$ where $D_k\defeq\diag(\pi_k)$.
  Since $\RWdown{k}$ is reversible, $D_k^{-1}\transpose{(\RWdown{k})} = \RWdown{k}D_k^{-1}$.
  Therefore,
  \begin{align}
    D\left(\tau\RWdown{k}~\|~\pi_k\right) & = \Ent{\pi_k}{D_{k}^{-1}\transpose{(\RWdown{k})}\transpose{\tau}}
    = \Ent{\pi_k}{\RWdown{k}D_{k}^{-1}\transpose{\tau}}\notag\\
    & \leq \Ent{\pi_{k-1}}{\Pup{k-1}D_{k}^{-1}\transpose{\tau}} \tag{by \Cref{lem:data-processing}} \\
    & \leq \left(1-\frac{1}{k}\right) \Ent{\pi_{k}}{D_{k}^{-1}\transpose{\tau}} \tag{by \Cref{lem:entropy-contraction}} \\
    & = \left(1-\frac{1}{k}\right) D\left(\tau~\|~\pi_k\right). \qedhere
  \end{align}
\end{proof}

It is well-known that the decay of relative entropy implies a mLSI.

\begin{proof}[Proof of \Cref{thm:main-complete}]
  Given any $f^{(k)}:\+M(k)\rightarrow\=R_{\ge 0}$ such that $\Ex_{\pi_k}f^{(k)}=1$, let $\tau = \transpose{(D_k f^{(k)})}$ be the distribution corresponding to $f^{(k)}$.
  Then,
  \begin{align*}
    & D\left(\tau~\|~\pi_k\right)-D\left(\tau\RWdown{k}~\|~\pi_k\right)\\
    =~& \sum_{S \in \+M(k)} \tau(S)\log\left(\frac{\tau(S)}{\pi_k(S)}\right)-\sum_{S \in \+M(k)} \tau \RWdown{k}(S)\log\left(\frac{\tau \RWdown{k}(S)}{\pi_k(S)}\right)\\
    =~& \sum_{S \in \+M(k)} \left[\tau\left(\identity-\RWdown{k}\right)\right](S)\log\left(\frac{\tau(S)}{\pi_k(S)}\right)-\sum_{S \in \+M(k)} \tau \RWdown{k}(S)\log\left(\frac{\tau \RWdown{k}(S)}{\tau(S)}\right)\\
    =~& \Diri{\RWdown{k}}{f^{(k)},\log f^{(k)}}-D\left(\tau \RWdown{k}~\|~\tau\right) \le \Diri{\RWdown{k}}{f^{(k)},\log f^{(k)}}.
  \end{align*}  
  Thus,
  \begin{align*}
    \Diri{\RWdown{k}}{f^{(k)},\log f^{(k)}}&\ge D\left(\tau~\|~\pi_k\right)-D\left(\tau\RWdown{k}~\|~\pi_k\right)\\
    &\ge \frac{1}{k}D\left(\tau~\|~\pi_k\right) = \frac{1}{k}\Ent{\pi_k}{D_k^{-1}\transpose{\tau}} \tag{by \Cref{cor:rel-entropy-contraction}}\\
    &= \frac{1}{k}\Ent{\pi_k}{f^{(k)}}.
  \end{align*}
  This proves the statement for $\RWdown{k}$. The same proof can be used for $\RWup{k}$ by replacing every occurrence of $\RWdown{k}$ with $\RWup{k}$, and the factor $\frac{1}{k}$ with $\frac{1}{k+1}$.  
\end{proof}



In fact, the contraction of relative entropy (\Cref{cor:rel-entropy-contraction}) directly implies the mixing time bound of \Cref{cor:mixing},
as illustrated by the following.

\begin{proof}[A direct proof of \Cref{cor:mixing}]
  We will only prove this for $\RWdown{k}$; the case of $\RWup{k}$ is similar. Notice that \Cref{cor:rel-entropy-contraction} implies that
  \begin{align*}
    D\left(\tau_0\left(\RWdown{k}\right)^{t} ~\big\|~ \pi_k \right) 
    & \leq \left(1-\frac{1}{k}\right)^{t}D\left(\tau_0~\|~\pi_k\right) \\
    & \leq e^{-t/k} D\left(\tau_0~\|~\pi_k\right) = e^{-t/k} \log \pi_k (x_0)^{-1},
  \end{align*}
  where $\tau_0$ is the initial distribution with $\tau_0(x_0)=1$ for some $x_0 \in \+M(k)$.
  Then, we use Pinsker's inequality 
  ($2\left\|\tau-\sigma\right\|_{\textnormal{TV}}^2 \leq D(\tau ~\|~  \sigma)$ for any
  two distributions~$\tau, \sigma$ on the same state space), to show
  \begin{align*}
    2\left\| \tau_0\left(\RWdown{k}\right)^{t} - \pi_k \right\|_{\textnormal{TV}}^2 
    \leq D\left(\tau_0\left(\RWdown{k}\right)^{t} ~\big\|~ \pi_k \right).
  \end{align*}
  Setting $e^{-t/k} \log \pi_k (x_0)^{-1} \leq 2\epsilon^2$, we conclude that
  \begin{align*}
    \left\|  \tau_0\left(\RWdown{k}\right)^{t} - \pi_k \right\|_{\textnormal{TV}} \leq \epsilon,
  \end{align*}
  whenever
  \begin{align*}
    t \geq k \left(\log \log \pi_k(x_0)^{-1} + \log \frac{1}{2\epsilon^2} \right).
  \end{align*}
  This gives us \Cref{cor:mixing} for $\RWdown{k}$.
\end{proof}

At the end of this section, let us comment that it is possible to prove the decay of variances similar to \Cref{lem:entropy-contraction},
with $\Ent{}{\cdot}$ replaced by $\Var{}{\cdot}$.
This provides an alternative proof for the spectral gap of $\RWBS$ to \citep{KO18,ALOV19}.
Indeed, the induction proof of \Cref{lem:entropy-contraction} does not require any change when one replaces $\Ent{}{\cdot}$ by $\Var{}{\cdot}$, as both of them obey the same decomposition rules.
However, the base case (namely \Cref{lem:k=2}) needs to be edited as follows.

\begin{lemma}  \label{lem:var_k=2}
  Let $f^{(2)}:\+M(2)\rightarrow\=R$.
  Then,
  \begin{align*}
    \Var{\pi_2}{f^{(2)}} \ge 2\Var{\pi_1}{f^{(1)}}.
  \end{align*}
\end{lemma}
\begin{proof}
    We begin by observing that
    \begin{align}\label{eqn:updown-decomp-1}
        \diag(w_1)\left(2\RWup{1}-\identity\right)=W_{\emptyset}.
    \end{align}
    From this identity and \Cref{prop:one-eigenvalue}, we deduce that the symmetric matrix $\diag(w_1)\left(2\RWup{1}-\identity\right)$ has at most one positive eigenvalue. Premultiplying by the positive semidefinite matrix $\diag(w_1)^{-1}$, we get that $2\RWup{1}-\identity$ also has at most one positive eigenvalue \citep[see, e.g.,][Lemma 2.6]{ALOV19}. Furthermore, the spectra of $2\RWup{1}-\identity$ and $2\RWdown{2}-\identity$ are the same up to some extra $-1$s. So, if $\abs{\+M(2)}\ge 2$ (otherwise the lemma holds trivially), $\lambda_2(\RWdown{2})\leq 1/2$ where $\lambda_2$ is the second largest eigenvalue. Then, the spectral gap $\lambda(\RWdown{2})=1-\lambda_2(\RWdown{2})\geq 1/2$, which means that
    \begin{align*}
        \Diri{\RWdown{2}}{f^{(2)},f^{(2)}}\ge \frac{1}{2}\Var{\pi_2}{f^{(2)}}.
    \end{align*}
    However, this is equivalent to the statement of the lemma, as can be seen by the following equalities:
    \begin{align*}
        \Var{\pi_1}{f^{(1)}}&=\transpose{\left(f^{(1)}\right)}D_1 f^{(1)}-\left(\Ex_{\pi_1}f^{(1)}\right)^2\\
        &=\transpose{\left(f^{(2)}\right)}\transpose{\left(\Pup{1}\right)}D_1 \Pup{1}f^{(2)}-\left(\Ex_{\pi_2}f^{(2)}\right)^2 \tag{by \Cref{lem:f-k}}\\
        &=\transpose{\left(f^{(2)}\right)}D_2\RWdown{2}f^{(2)}-\left(\Ex_{\pi_2}f^{(2)}\right)^2 \tag{by \eqref{eqn:Pup-Pdown}}\\
        &=\Var{\pi_2}{f^{(2)}}-\Diri{\RWdown{2}}{f^{(2)},f^{(2)}}. \qedhere
    \end{align*}
\end{proof}

\section{Concentration}\label{sec:concentration}

One application of the modified log-Sobolev inequalities is to show concentration inequalities,
via the Herbst argument \citep[see, e.g.,][]{BT06,BLM13}.
In the discrete setting, concentration inequalities have been obtained by \citet[Section 5]{Goel04} and can also be obtained by combining various results by \cite{BG99,Sam05,BHT06}.
The following lemma and its proof are a small modification of \citep[Lemma 5]{HS19}.
For completeness, we include all details.

\begin{lemma}  \label{lem:herbst-argument}
  Let $P$ be the transition matrix of a reversible Markov Chain with stationary distribution $\pi$ on a finite set $\Omega$, and $f:\Omega\rightarrow\=R$ be some observable function.
  Then,
  \begin{align*}
    \Pr_{x\sim\pi}(f(x)-\Ex_{\pi} f \geq a) ~\leq~ \exp\left(-\frac{\rho(P)a^2}{2v(f)}\right),
  \end{align*}
  where $a \geq 0$ and
  \begin{align*}
    v(f) := \max_{x \in \Omega}\left\{\sum_{y \in \Omega}P(x,y)(f(x)-f(y))^2\right\}.
  \end{align*}
\end{lemma}

\begin{proof}
  For any $x\in \Omega$ and $t\in(0,+\infty)$, let
  \begin{align*}
    F_t(x) \defeq \exp\left(tf(x)-ct^2\right),
  \end{align*}
  where $c\defeq\frac{v(f)}{2\rho(P)}$.
  We will use the inequality
  \begin{align}\label{eqn:cosh}
    z(e^z+1) \geq 2(e^z-1),
  \end{align}
  which holds for $z\geq 0$. 
  To see this, notice that at $z=0$ the equality holds, 
  and for $z>0$ the derivative of the left is larger than that of the right.
  
  If $f(x)\geq f(y)$, we set $z = t(f(x)-f(y))$ in \eqref{eqn:cosh} and obtain
  \begin{align*}
    t(f(x)-f(y))(F_t(x)+F_t(y))\geq 2(F_t(x)-F_t(y)),
  \end{align*}
  which in turn implies that
  \begin{align}\label{eqn:F}
    (F_t(x)-F_t(y))(f(x)-f(y))\leq \frac{t}{2}(F_t(x)+F_t(y))(f(x)-f(y))^2.
  \end{align}
  Notice that \eqref{eqn:F} also holds even if $f(x)<f(y)$ by swapping $x$ and $y$.
  Thus, we have that
  \begin{align*}
    \Diri{P}{F_t, \log F_t} & = \frac{t}{2}\sum_{x\in \Omega}\sum_{y\in \Omega}\pi(x)P(x,y)(F_t(x)-F_t(y))(f(x)-f(y)) 
    \tag{by \eqref{eqn:Dirichlet-reversible}}\\
    & \leq \frac{t^2}{4}\sum_{x\in \Omega}\sum_{y\in \Omega}\pi(x)P(x,y)(F_t(x)+F_t(y))(f(x)-f(y))^2 \tag{by \eqref{eqn:F}} \\
    & = \frac{t^2}{2}\sum_{x\in \Omega}\pi(x)F_t(x)\sum_{y\in \Omega}P(x,y)(f(x)-f(y))^2 \tag{by the reversibility of $P$} \\
    & \leq \frac{t^2}{2}v(f)\Ex_{\pi} F_t.
  \end{align*}
  This, together with $\Diri{P}{F_t, \log F_t}\geq \rho(P)\Ent{\pi}{F_t}$ (recall the definition of $\rho(P)$), yields
  \begin{align*}
    \Ent{\pi}{F_t}\leq ct^2\Ex_{\pi} F_t.
  \end{align*}
  By noticing that
  \begin{align*}
    \frac{d}{dt}\left(\frac{log\Ex_{\pi}F_t}{t}\right) = \frac{\Ent{\pi}{F_t}-ct^2\Ex_{\pi} F_t}{t^2\Ex_{\pi} F_t} \leq 0,
  \end{align*}
  we deduce that for any $t>0$,
  \begin{align*}
    \frac{log\Ex_{\pi}F_t}{t} \leq \lim_{\hspace{.15cm}h \to 0^+}\frac{log\Ex_{\pi}F_h}{h} = \Ex_{\pi}f,
  \end{align*}
  or equivalently,
  \begin{align*}
    \Ex_{\pi}F_t & \le \exp\left( t\Ex_{\pi}f \right).
  \end{align*}
  Finally, by Markov inequality, for any $a\geq 0$,
  \begin{align*}
    \Pr_{x\sim\pi}(f(x)-\Ex_{\pi} f \geq a) & =~ \Pr_{x\sim\pi}\big(F_t(x)\ge\exp(t\Ex_{\pi}f-ct^2+ a t)\big)\\
      & \le~ \exp\left(ct^2- a t\right),
  \end{align*}
  where the right hand side is minimized for $t=\frac{a}{2c}=\frac{a\rho(P)}{v(f)}$.
\end{proof}

\Cref{cor:concentration} follows from applying \Cref{lem:herbst-argument} to both $f$ and $-f$ together with \Cref{thm:main}.
We could also apply \Cref{lem:herbst-argument} together with \Cref{thm:main-complete} to get concentration inequalities for all $\pi_k$.

For a Lipschitz function $f:\Omega\rightarrow\=R$ with Lipschitz constant $c$ (under the graph distance in the bases-exchange graph), 
we have that $v(f) \leq c^2$. 
Thus, by \Cref{cor:concentration}, such a Lipschitz function satisfies the following concentration inequality:
\begin{align*}
  \Pr_{x\sim\pi}(\abs{f(x)-\Ex_{\pi} f} \geq a) ~\leq~ 2\exp\left(-\frac{a^2}{2rc^2}\right),
\end{align*}
when $\pi$ is an $r$-homogeneous strongly log-concave distribution.

For general matroids, an example is the function that counts the number of elements belonging to a specified subset of the ground set, which has Lipschitz constant $c=1$. 
More examples were given by \cite{pemantle_peres_2014} for graphic matroids,
such as functions that count the number of leaves in a spanning tree ($c=2$),
or the number of vertices with odd degrees ($c=4$).

\section*{Acknowledgements}

We thank Mark Jerrum and Prasad Tetali for their helpful comments,
and Arthur Sinulis for pointing out a saving of a factor $2$ in \Cref{cor:concentration}.
We also thank the anonymous referees whose comments have resulted in an improvement of the presentation.

Part of the work was done while HG was visiting the Simons institute for the theory of computing in the University of California --- Berkeley.

\appendix
\section{Stochastic covering property and strong log-concavity}
\label{sec:SCP}

The results obtained by \cite{pemantle_peres_2014} and \cite{HS19} only require a property which is weaker than the strong Rayleigh property (\SRP), 
namely the \emph{stochastic covering property} (\SCP).
Since strong log-concavity (\SLC) is also a generalisation of \SRP,
it is natural to wonder about the relationship between \SLC{} and \SCP.
In this section we show that \SLC{} is incomparable to \SCP.
As a result, \cref{thm:main} and \cref{cor:concentration} do not subsume the results of \cite{HS19} and \cite{pemantle_peres_2014}, respectively.
Moreover, \cref{cor:concentration} is also incomparable to the concentration bound of \cite{GV18},
whose result requires only \emph{negative regression}, a property weaker than \SCP.

First, let us define \SCP. 
For $S\subseteq [n]$ and $x,y\in \{0,1\}^S$, we say $x$ \emph{covers} $y$, denoted by $x\covers y$,
if $x=y$ or $x=y+\*e_i$ for some $i$, where $\*e_i$ is the unit vector of coordinate $i$.
In other words, $x$ is obtained from $y$ by increasing at most one coordinate.
For two distributions $\mu$ and $\nu$, we say $\mu$ \emph{stochastically covers} $\nu$,
if there is a coupling such that $\Pr_{X\sim\mu,Y\sim\nu}(X\covers Y)=1$.
With slight overload of notation, we also write $\mu\covers\nu$.
A distribution $\mu:\{0,1\}^{[n]}\rightarrow \=R_{\ge 0}$ satisfies the \SCP{} 
if for any $S\subseteq [n]$ and $x,y\in \{0,1\}^S$ such that $x\covers y$,
$\mu_y\covers \mu_x$,
where $\mu_x$ is the distribution of $\mu$ conditioned on agreeing with $x$ over the index set $S$.

Furthermore, $\mu$ is said to satisfy the \emph{negative cylinder dependence} (\NCD),
if for any $S\subseteq [n]$,
\begin{align*}
  \Ex \prod_{i\in S} X_i & \le \prod_{i\in S} \Ex X_i,\\
  \Ex \prod_{i\in S}(1-X_i) & \le \prod_{i\in S} \Ex (1-X_i),
\end{align*}
where $X_i$ is the indicator variable of coordinate $i$.
It is known that \SCP{} implies \NCD{} \citep{pemantle_peres_2014}.
However, such negative dependence even when $\abs{S}=2$ is known \emph{not} to hold for the uniform distribution over the bases of some matroids.
See \citep{HSW18} for the most comprehensive list of such examples that we are aware of.
As the uniform distribution over a matroid's bases is \SLC,
\SLC{} does not imply \SCP.

On the other hand, \SCP{} does not imply \SLC{} either.
We give a concrete example here.
Let $\mu$ be supported on the bases of the uniform matroid of rank $2$ over $4$ elements.
We choose $\mu$ such that
\begin{align*}
  \mu(\{1,1,0,0\})&\propto \theta, & \mu(\{1,0,1,0\})&\propto 2, & \mu(\{1,0,0,1\})&\propto 1, \\
  \mu(\{0,1,1,0\})&\propto 1, & \mu(\{0,1,0,1\})&\propto 1, & \mu(\{0,0,1,1\})&\propto 1.  
\end{align*}
It is straightforward to verify that if $0\le \theta<3-2\sqrt{2}\approx 0.17157$ or $\theta>3+2\sqrt{2}\approx 5.82843$, then \SLC{} fails.
However, \SCP{} holds as long as $0\le \theta\le 6$.
To see the latter claim, first verify that the distribution conditioned on choosing any $i\in[4]$ stochastically dominates the one conditioned on not choosing $i$.
Then notice that in a homogeneous distribution, such stochastic dominance is the same as stochastic covering.

Here is some insight on how to find an example such as the above.
When the generating polynomial $g_{\mu}$ is homogeneous and quadratic, 
it is \SLC{} if and only if it has the \SRP{} \citep{BH19},
which in turn is equivalent to the following condition as $g_{\mu}\in\=R[x_1,\dots,x_n]$ is multiaffine:
\begin{align}\label{eqn:ND}
  \frac{\partial}{x_i} g_{\mu}(\*x)\cdot\frac{\partial}{x_j} g_{\mu}(\*x) \ge g_{\mu}(\*x)\cdot\frac{\partial^2}{\partial x_i\partial x_j} g_{\mu}(\*x),
\end{align}
for any $i,j\in[n]$ and $\*x\in\=R^n$.
See \citep{Bra07}.
If we plug in $\*x=\*1$, then \eqref{eqn:ND} becomes negative dependence for a pair of variables,
which is a special case of \NCD{} and thus a necessary condition for \SCP.
In our example, we choose $\mu$ so that \eqref{eqn:ND} holds for $\*x=\*1$ but not for an arbitrary $\*x\in\=R^n$.
It turns out that our choice is also sufficient for \SCP{} in this particular setting.

\bibliographystyle{plainnat} 
\bibliography{Log-Sob}

\end{document}